\documentclass[12pt]{amsart}
\usepackage{amsmath,amssymb,amsbsy,amsfonts,latexsym,amsopn,amstext,
                                               amsxtra,euscript,amscd,bm}
\usepackage[margin=1in]{geometry}                    
\usepackage{url}
\usepackage[colorlinks,linkcolor=blue,anchorcolor=blue,citecolor=blue]{hyperref}
\usepackage{color}
\usepackage{enumerate}

\begin{document}

\newtheorem{problem}{Problem}
\newtheorem{theorem}{Theorem}
\newtheorem{lemma}[theorem]{Lemma}
\newtheorem{claim}[theorem]{Claim}
\newtheorem{corollary}[theorem]{Corollary}
\newtheorem{prop}[theorem]{Proposition}
\newtheorem{definition}{Definition}
\newtheorem{question}[theorem]{Question}
\newtheorem{application}[theorem]{Application}
\newtheorem{remark}[theorem]{Remark}

\def\F{{\mathbb F}}
\def\R{{\mathbb R}}


\title[Szemer\'{e}di-Trotter type results in $\F_q$]{Szemer\'{e}di-Trotter type results in arbitrary finite fields}

\author{Ali Mohammadi}

\address{School of Mathematics and Statistics, University of
Sydney, NSW 2006, Australia}
\email{alim@maths.usyd.edu.au}

\pagenumbering{arabic}

\begin{abstract}
Let $q$ be a power of a prime and $\F_q$ the finite field consisting of $q$ elements. We prove explicit upper bounds on the number of incidences between lines and Cartesian products in $\F_q^2$. We also use our results on point-line incidences to give new sum-product type estimates concerning sums of reciprocals.
\end{abstract}

\maketitle

\section{Introduction}
Let $F$ be an arbitrary field. Given a finite set of points $P$ and a finite set of lines $L$ in the plane $F^{2}$, we define the number of incidences between $P$ and $L$ by
\[
 I(P,L) = |\{(p,l)\in P \times L : p\in l\}|.
\]
An elementary argument, which involves an application of the Cauchy-Schwarz inequality, yields the trivial bound
\begin{equation}
\label{eqn:IPLTB}
 I(P, L) \leq \min \{ |P|^{1/2}|L| + |P|, |L|^{1/2}|P| + |L| \}.
\end{equation}
See \cite[Corollary~5.2]{BouKatTao} for a proof of the above inequality. In particular, in the critical case $|P| = |L| = N$, we have $I(P, L) \leq N^{3/2}$.

In the case $F = \mathbb{R}$, Szemer\'{e}di and Trotter~\cite{SzTr} proved the bound
\[
I(P, L) \ll |P|^{2/3}|L|^{2/3} + |P| + |L|.
\]
A construction due to Elekes~\cite{Ele} demonstrates that this bound is sharp up to constants.

Let $p$ be a prime, $\F_q$ the finite field consisting of $q = p^{m}$ elements and $\F_q^* = \F_q\setminus\{0\}$. The primary purpose of this paper is to establish nontrivial upper bounds on $I(P, L)$ for sets of points $P$ and lines $L$ in $\F^{2}_q$. An immediate obstacle in this setting is the presence of nontrivial subfields. Given a subfield $G$ of $\F_q$, let $P = G \times G$ and let $L =\{l_{a, b}: a,b \in G\}$, where
$l_{a, b} = \{(x, y) \in \F_q^2: y = ax + b\}.$ Then $|P| = |G|^2$, $|L| \approx |G|^2$ and $I(P, L) \approx |G|^3$. Therefore, in this example, the bound \eqref{eqn:IPLTB} is optimal up to constants. We deduce that such point sets, as well as their affine transformations, must be avoided in order to ensure a nontrivial incidence bound holds.

Let $X(P) = \{x : (x,y) \in P\}.$ Jones \cite{Jones} proved that there exists an absolute constant $\gamma>0$ such that the bound 
\begin{equation}
\label{eqn:JonesIncidenceBound}
I(P, L) \ll N^{3/2 - 1/12838}
\end{equation}
holds if $P$ satisfies the following two conditions.
\begin{enumerate}[(i)]
\item
For all subfields $G$ and elements $c,d$ in $\F_q$,
\[
|X(P) \cap (cG + d)| \leq \max \{|G|^{1/2}, \gamma N^{2560/6419}\}.
\]
\item
For every subfield $G$ in $\F_q$ with $|G| \geq \gamma N^{2560/6419}$, $X(P)$ intersects strictly fewer than $\max\{|G|^{1/2},\gamma N^{2560/6419}\}/2$ distinct translates $G+d$ of $G$.
\end{enumerate}
Clearly, the second condition is undesirable. We remark that the proof of the above result can be reworked to drop this condition if one restricts to certain structured point sets such as Cartesian products. Indeed this is the approach we consider in this paper, noting that it is significantly more difficult to obtain results of similar strength for general point sets.

For large sets of points and lines, with $|P|, |L| \geq q$, Vinh \cite{Vinh} proved the bound
\begin{equation}
\label{eqn:VinhIPL}
\bigg| I(P, L) - \frac{|P||L|}{q}\bigg| \leq (q|P||L|)^{1/2}.
\end{equation}
Also see \cite{MurPet} for an elementary proof of this result.

In the regime of small sets, Stevens and de Zeeuw \cite{StevZeeuw} proved the following result. Let $A, B \subset F$, with $|B| \leq |A|$ and let $L$ be a set of lines in $F^2$ such that $|L|^{3} \geq |A|^{2}|B|$. If $F$ has positive characteristic $p$, assume $|B||L| \ll p^{2}$. Then 
\[
I(A\times B, L) \ll |A|^{1/2}|B|^{3/4}|L|^{3/4} + |L|.
\]
As outlined in \cite[Example~5]{StevZeeuw}, this bound is optimal for certain sets of points and lines. However, in the setting $F = \F_q$,  if $q$ is a large power of $p$ then the above estimate becomes restricted to very small sets of points and lines. The main ingredient in the proof of the above result is a bound on the number of incidences between points and planes due to Rudnev~\cite{Rud2}. Indeed, \cite{Rud2} has been the driving force behind much of the recent progress on sum-product type problems as well as certain related geometric questions in fields of positive characteristic. A survey on some of such developments is provided in \cite{Rud3}. However, we note that these new techniques appear to be ineffective in dealing with general subsets of $\F_q$.

Prior to the appearance of Rudnev's result on point-plane incidences, sum-product estimates provided the key tool in the study of point-line incidences in finite fields. For sets $A, B \subset \F_q$ we define the sum set $A+B = \{a+b: a\in A, b \in B\}$ and similarly define the difference set $A- B$, the product set $AB$ and the ratio set $A/B$. To avoid dividing by zero, it is assumed throughout the paper that all sets contain strictly more than one element and exclude zero.

Bourgain, Katz and Tao~\cite{BouKatTao} proved the first qualitative sum-product estimate in $\F_p$, which states that for any $\alpha >0$ there exists a $\beta = \beta(\alpha) > 0$ such that if $A\subset \F_p$ and $p^{\alpha} < |A| < p^{1-\alpha}$, then 
\[
\max\{|A+A|, |AA|\} \gg_{\alpha} |A|^{1+\beta}.
\]
As a corollary, they proved that, for any $0 < \epsilon < 2$ there exists a $\delta = \delta(\epsilon) > 0$ such that, given a set of points $P$ and a set of lines $L$ in $\F_p^{2}$, if $|P|, |L| \leq N = p^{\epsilon}$ then
\begin{equation}
\label{eqn:IPLBKT}
I(P, L) \ll N^{3/2 - \delta}.
\end{equation}
This result is based, roughly, on the observation that if there are too many incidences between elements of $P$ and $L$, then one can identify a set $A \subset \F_p$, of cardinality close to $|P|^{1/2}$, such that $|A+A|$ and $|AA|$ are both small. Consequently, through the use of a sum-product estimate, one concludes that $A$ must occupy nearly all of $\F_p$.

In a similar vein, building on the subsequent progress on the sum-product problem, Helfgott and Rudnev~\cite{HelRud} and later Jones~\cite{Jones2} obtained explicit variants of the above incidence result.

In this paper, we use sum-product estimates, as well as related techniques, to establish nontrivial upper bounds on the number of collinear triples $T(A, B)$ formed by a Cartesian product $A\times B$, where $A, B \subset \F_q$. Then, in the same manner as \cite[Corollary~6]{AkMuRuSh}, we use the acquired estimates of $T(A, B)$ to bound $I(A\times B, L)$ given any set of lines $L$ in $\F_q^{2}$. Formally, we define $T(A, B)$ as the number of sextuples $(a_1, a_2, a_3, b_1, b_2, b_3) \in A^{3}\times B^{3}$ satisfying
\begin{equation}
\label{triplesDefn}
(a_2 - a_1)(b_3 - b_1) = (a_3 - a_1)(b_2 - b_1).
\end{equation}
For the sake of simplicity, we write $T(A)$ instead of $T(A, A)$. We observe the trivial upper bound
\begin{equation}
\label{eqn:TriplesTrivialBound}
T(A, B) \leq |A|^{2}|B|^{2}\cdot \min\{|A|, |B|\}.
\end{equation}
Clearly, for all $c_1, c_2 \in \F_q^{*}$ and all $d_1, d_2 \in \F_q$, we have
\[
T(c_1A + d_1, c_2B + d_2) = T(A, B).
\]
Additionally, note that \eqref{eqn:TriplesTrivialBound} becomes an equality if $A = B = G$ for some subfield $G$ of $\F_q$. However, one expects that a nontrivial upper bound on $T(A, B)$ holds as long as either $A$ or $B$ does not correlate with any sets of the form $cG+d$, for subfields $G$ and elements $c,d \in \F_q$. 

Let $T^{*}(A, B)$ denote the number of nontrivial collinear triples of $A\times B$, defined as the number of sextuples $(a_1, a_2, a_3, b_1, b_2, b_3) \in A^{3}\times B^{3}$ satisfying
\begin{equation}
\label{Def:NTTriples}
(a_2 - a_1)(b_3 - b_1) = (a_3 - a_1)(b_2 - b_1)\not= 0.
\end{equation}
Then, assuming $|B|\leq |A|$, it follows that
\begin{equation}
\label{eqn:T*toT}
T(A,B) \approx T^{*}(A, B) + |A|^{3}|B|.
\end{equation}
The term $|A|^{3}|B|$ can be interpreted as the contribution to the number of collinear triples coming from $|B|$ horizontal lines. It is worth noting that if $|B| \ll |A|^{1/2}$ then \eqref{eqn:TriplesTrivialBound} and \eqref{eqn:T*toT} together imply $T(A, B) \approx |A|^{3}|B|$.

We mention that for $A \subset \F_q$ with $|A| \ll p^{2/3}$, Aksoy~Yazici et al.~\cite{AkMuRuSh} proved the bound $T(A) \ll |A|^{9/2}.$ Furthermore, $T(A)$ has been studied extensively by Murphy et al.~\cite{MuPeRoRu} in the context of prime fields.

Define $L(P)$ to be the number of distinct lines determined by pairs of points of $P \subset \F_q^2$. As a further application of our bounds on $T(A, B)$, we prove nontrivial lower bounds on $L(P)$ for Cartesian products $P = A\times B,$ with $A, B \subset \F_q$. See \cite{AkMuRuSh, StevZeeuw} for stronger estimates which hold for sets of cardinality close to the characteristic $p$.

Finally, we use our bounds on point-line incidences to prove explicit lower bounds on the quantities $\max\{|A+A|, |1/A + 1/A|\}$ and $|A + 1/A|$ for sets $A \subset \F_q$ which are not close in size to any proper subfields.

\subsection*{Asymptotic notation}
For positive real numbers $X$ and $Y$, we write $X \ll Y$, $Y \gg X$, $X = O(Y)$ and $Y = \Omega(X)$ to all mean that $X \leq cY$ for some absolute constant $c > 0$. If the constant $c$ depends on a parameter $\epsilon$, we write $X = O_{\epsilon}(Y)$ etc. If $X \ll Y$ and $Y \ll X$, we write $X  = \Theta(Y)$ or $X \approx Y$. We also use $X \lesssim Y$ to denote the existence of an absolute constant $c >0$, such that $X \ll (\log Y)^{c} Y$.

\section{Main results}
\subsection{Incidence bounds}
Our first result is based on some techniques introduced in \cite{HelRud}, which also underlie the incidence result of Jones~\cite{Jones}. However, our approach differs from \cite{Jones} in that we consider only Cartesian products, thereby relaxing the required constraints on the point sets.
\begin{theorem}
\label{thm:UnbalancedIncidences}
Let $A, B \subset \F_q$ with $|B| \leq |A|$ and let $L$ be a set of lines in $\F_q^2$. Suppose that 
\[
|A \cap (cG + d)| \ll \max \big\{|G|^{1/2}, |A|^{31/191}|B|^{129/191}\big\}
\]
for all proper subfields $G$ of $\F_q$ and all elements $c, d \in \F_q$. Then
\begin{equation}
 \label{eqn:TriplesUnbalanced}
T(A, B) \ll  |A|^{383/191}|B|^{571/191} + q^{-1/103}|A|^{207/103}|B|^{3} + |A|^{3}|B|,
\end{equation}
\begin{equation}
\label{eqn:UnbalancedcartesianIncidence1}
I(A \times B, L) \ll \big(|A|^{383/573}|B|^{571/573} + q^{-1/309}|A|^{207/309}|B| + |A||B|^{1/3}\big)|L|^{2/3} + |L|,
\end{equation}
\begin{equation}
\label{eqn:NoOfLinesUnbalanced}
L(A \times B) \gg \min \big\{|A|^{380/191}|B|^{4/191}, q^{2/103}|A|^{204/103}, |B|^{4} \big\}.
\end{equation}
\end{theorem}

In the case $B = A$, we obtain the following improvement of Theorem~\ref{thm:UnbalancedIncidences}.
\begin{theorem}
\label{thm:balancedIncidences}
Let $A\subset \F_q$ and let $L$ be a set of lines in $\F_q^2$. Suppose that 
\begin{equation}
\label{eqn:TriplesRestriction}
|A \cap (cG + d)| \ll \max \big\{|G|^{1/2}, |A|^{51/52}\big\},
\end{equation}
for all proper subfields $G$ of $\F_q$ and all elements $c, d \in \F_q$. Then, we have the estimates
\begin{equation}
 \label{eqn:Triplesbalanced}
T(A) \ll |A|^{5-1/104} + q^{-1/95}|A|^{5+1/95},
\end{equation}
\begin{equation}
\label{eqn:cartesianIncidence1}
I(A \times A, L) \ll \big(|A|^{173/104} + q^{-1/285}|A|^{476/285}\big) |L|^{2/3} + |L|,
\end{equation}
\begin{equation}
\label{eqn:NoOfLines}
L(A \times A) \gg \min\{|A|^{2+1/52}, q^{2/95}|A|^{2-2/95}\}.
\end{equation}
\end{theorem}
To compare the above result with \eqref{eqn:JonesIncidenceBound}, let $P =A \times A$ for a set $A \subset \F_q$, which satisfies restriction \eqref{eqn:TriplesRestriction}. Then given any set of lines $L$ in $\F_q^{2}$, if $|L| = |P| = N$, by \eqref{eqn:cartesianIncidence1} we have
\[
I(P, L) \ll N^{3/2 - 1/624} + q^{-1/285}N^{3/2 + 1/570}.
\]
This also improves on \eqref{eqn:VinhIPL} in the range $N < q^{1+ 1/311}$. Given any sets $A, B \subset \F_q$, with $|B| \leq |A|$, by the Cauchy-Schwarz inequality we have
\begin{equation}
\label{eqn:TCS}
T(A, B) \ll T(A)^{1/2}T(B)^{1/2} + |A|^{3}|B|.
\end{equation}
Suppose that the sets $A$ and $B$ both satisfy condition \eqref{eqn:TriplesRestriction} and $|A| \ll q^{1/2}$. Then by Theorem~\ref{thm:balancedIncidences}, together with inequality \eqref{eqn:TCS}, we obtain
\begin{equation}
\label{eqn:TriplesViaCS}
T(A, B) \ll |A|^{519/208}|B|^{519/208} + |A|^{3}|B|.
\end{equation}
It follows that estimate \eqref{eqn:TriplesUnbalanced} is stronger than \eqref{eqn:TriplesViaCS} in the range
\[
|A|^{1/2} < |B| \ll |A|^{1-c},
\]
where $c = 174/19639 < 1/112$.

Our next result can be used to obtain nontrivial upper bounds on $T(A)$, for sets $A\subset \F_q$, if either $|A+A|$ or $|A-A|$ is small.
\begin{theorem}
\label{Thm:TSumsetBound}
Let $A \subset \F_q$. Suppose that there exists some $\delta >0$ such that
\begin{equation}
\label{eqn:TSmallSumsetRestriction}
|A \cap (cG + d)| \ll \max\{|G|^{1/2}, |A|^{1-\delta}\}
\end{equation}
for all proper subfields $G$ of $\F_q$ and all elements $c, d \in \F_q$. Then
\begin{equation}
\label{eqn:SumsetToTriples}
T(A) \ll \log|A| \cdot \max\{|A + A|^{7/4}|A|^{3}, |A + A|^{6/5}|A|^{18/5}, |A|^{5-\delta/2}, |A + A|^{7/4}|A|^{7/2}q^{-1/4} \}.
\end{equation}
The same estimate holds with the sum set $A + A$ replaced by the difference set $A -A$. 
\end{theorem}
In particular, assuming $|A+A| \approx|A|$, we obtain significant improvements over the estimates of Theorem~\ref{thm:balancedIncidences}.
\begin{corollary}
\label{cor:SmSSIPL}
Let $A \subset \F_q$. Suppose that $|A+A| \approx |A|$ and 
\begin{equation}
|A\cap (cG + d)| \ll \max\{|G|^{1/2}, |A|^{3/5}\}
\end{equation}
for all elements $c, d\in \F_q$ and proper subfields $G$. Then, for any set of lines $L$ in $\F_q^{2}$, we have the estimates
\begin{equation}
\label{eqn:CTIntervals}
T(A) \ll \log|A| \cdot (|A|^{5-1/5} + q^{-1/4}|A|^{5+1/4}),
\end{equation}
\begin{equation}
\label{eqn:IncidencesIntervals}
I(A\times A, L) \ll \log|A|^{1/3}\cdot (|A|^{24/15} + q^{-1/12}|A|^{21/12})|L|^{2/3} + |L|,
\end{equation}
\begin{equation}
\label{eqn:LinesIntervals}
L(A\times A) \gg \log|A|^{-2}\cdot \min\{|A|^{2+2/5}, q^{1/2}|A|^{2-1/2}\}.
\end{equation}
\end{corollary}

\subsection{Applications}
Based on Theorem~\ref{thm:balancedIncidences}, we obtain the following result which provides an explicit variant of \cite[Theorem 4]{Bou} for subsets of $\F_q$. Also see \cite{BouGar2} for sharp bounds on sums of reciprocals of intervals in $\F_p$.
\begin{corollary}
\label{cor:SR}
Let $A, B \subset \F_q$. Suppose that
\begin{equation}
\label{eqn:ERCon}
|(A+B)^{-1} \cap (cG+d)| \ll \max\{|G|^{1/2}, |A+B|^{51/52}\},
\end{equation}
for all proper subfields $G$ and elements $c, d$ in $\F_q$. Then
\begin{equation}
\label{eqn:ERABFq}
E_{+}(1/A, 1/B) \ll \Big(|A + B|^{173/104} + q^{-1/285}|A+B|^{476/285}\Big) |B|^{4/3}.
\end{equation}
Consequently, if condition \eqref{eqn:ERCon} holds with $B=A$, then
\begin{equation}
\label{eqn:SRvsSS}
\max\{|A+A|, |1/A + 1/A|\} \gg \min\{|A|^{1+1/831}, q^{1/761}|A|^{1-1/761}\}.
\end{equation}
If condition \eqref{eqn:ERCon} holds with $B = A^{-1}$, then
\begin{equation}
\label{eqn:SRExpander}
|A + 1/A| \gg \min\{|A|^{1+1/831}, q^{1/761}|A|^{1-1/761}\}.
\end{equation}
Alternatively, estimates \eqref{eqn:SRvsSS} and \eqref{eqn:SRExpander} hold if the cardinality of $A$ does not lie in the intervals  $(|G|^{1/2 - 1/1664}, |G|^{1+1/51})$ for all proper subfields $G$ of $\F_q$.
\end{corollary}

For a set $A \subset \F_q$, with a small sum set,  we use Corollary~\ref{cor:SR} to obtain a nontrivial upper bound on the number of solutions to the hyperbola $xy = \alpha$, where $(x, y) \in A \times A$. See \cite[Corollary 15]{AkMuRuSh} for a stronger analogue of this result which holds if $|A| < p^{5/8}$. Also see \cite{CilGar} for sharp estimates concerning intervals in $\F_p$.
\begin{corollary}
\label{cor:Hyperbolic}
Let $A \subset \F_q$. Suppose that 
\begin{equation}
\label{eqn:HypSolCon}
|(A+A)^{-1} \cap (cG+d)| \ll \max\{|G|^{1/2}, |A+A|^{47/48}\}
\end{equation}
for all proper subfields $G$ and elements $c, d$ of $\F_q$. Then, for any $\alpha \in \F_q^*$, we have
\begin{equation}
\label{eqn:HyperbolaEstimate}
|A \cap \alpha/A| \ll |A + A|^{1 - 1/832} + q^{-1/760}|A+A|^{1+1/760}.
\end{equation}
Alternatively, estimate \eqref{eqn:HyperbolaEstimate} holds if $|A| \not\in (|G|^{1/2 -1/1664}, |G|^{1+1/47})$ for all proper subfields $G$ of $\F_q$.
\end{corollary}

We use Theorem~\ref{Thm:TSumsetBound} to obtain the following improvement of estimate \eqref{eqn:SRvsSS} for additive groups.
\begin{corollary}
\label{cor:SRI}
Let $A\subset \F_q$ be an additive group. Suppose that
\begin{equation}
\label{eqn:SoRICondition}
|A^{-1} \cap (cG +d)| \ll \max\{|G|^{1/2}, |A|^{4/7}\}
\end{equation}
for all subfields $G$ and elements $c, d$ in $\F_q$.
Then
\begin{equation}
\label{eqn:SoRInterval}
|1/A \pm 1/A| \gg (\log|A|)^{-1/3}\cdot \min\{|A|^{1+1/21}, q^{1/19}|A|^{1-1/19}\}.
\end{equation}
Alternatively, estimate \eqref{eqn:SoRInterval} holds if $|A| \not\in (|G|^{1/2}, |G|^{1+3/7})$ for all proper subfields $G$ of $\F_q$.
\end{corollary}

\section{Preliminaries}
We require a basic extension of a sum-product type estimate due to Roche-Newton \cite{Roche}. Since \cite{Roche} has not been peer-reviewed, we provide a full proof of this result in Appendix~\ref{Appendix:SumRatio}, which closely follows the original arguments.
\begin{lemma}
\label{lem:sum-ratio}
Let $A \subseteq \F_q$ and let $0 < \eta < 1/8$. Suppose $|A| \ll q^{1/2}$ and that
\begin{equation}
\label{eqn:SumRatioCondition}
 |A \cap cG| \leq \max \big\{C|G|^{1/2},\eta|A|\big\} 
\end{equation}
for all proper subfields $G$ of $\F_q$, elements $c\in \F_q$ and some constant $C > 0$. Then either
\[ 
|A \pm A|^7|A/A|^4 \gg_{\eta} |A|^{12} \qquad \text{or} \qquad |A \pm A|^6|A/A|^5 \gg_{\eta} |A|^{12}.
\]
If $|A| > \eta^{-1}q^{1/2}$, irrespective of condition \eqref{eqn:SumRatioCondition}, we have
\[
 |A \pm A|^7|A/A|^4 \gg_{\eta} |A|^{10}q.
\]
\end{lemma}
\begin{remark}
Following a similar approach as \cite[Theorem~1]{Gar2} or \cite[Corollary~2]{Vinh}, for any set $A\subset \F_q$, one can establish the bound
\[
\max\{|A\pm A|, |A/A|\} \gg \min\{|A|^{1/2}q^{1/2}, |A|^{2}/q^{1/2}\}.
\]
This bound is nontrivial if $|A| > q^{1/2}$ and, as demonstrated in \cite{Gar2}, it is optimal up to constants if $|A| >q^{2/3}$. We point out that Lemma~\ref{lem:sum-ratio} improves on this bound if $|A| < q^{13/24}$.
\end{remark}
Given sets $A, B \subseteq \F_q$, we define the multiplicative energy between $A$ and $B$ by
\begin{equation}
\label{eqn:MEnergyDefn}
E_{\times}(A, B) = |\{ (a_1, a_2, b_1, b_2) \in A^{2}\times B^{2} : a_1 b_1 = a_2 b_2\}|.
\end{equation}
We write simply $E_{\times}(A)$ instead of $E_{\times}(A, A)$. Using the Cauchy-Schwarz inequality, we get
\begin{equation}
\label{eqn:MEECS}
E_{\times}(A, B) \leq E_{\times}(A)^{1/2}E_{\times}(B)^{1/2}.
\end{equation}
See \cite[Corollary~2.10]{TaoVu} for a proof of the above inequality. 

One may recover a bound on the multiplicative energy of subsets of $\F_q$ from the proof of~\cite[Theorem~1.4]{RoLi}. We state a slightly generalised version of this bound below and give a sketch of the proof in Appendix~\ref{Appendix:SumRatio}.
\begin{lemma}
\label{lem:energybound}
Let $A \subseteq \F_q$. Suppose that
\begin{equation}
\label{eqn:EnergyBoundRestriction}
 |A \cap cG| \ll \max \big\{|G|^{1/2}, |A|^{1-\delta} \big\} 
\end{equation}
for all proper subfields $G$ of $\F_q$, elements $c\in \F_q$ and some fixed $\delta > 0$. Then
\begin{equation}
\label{eqn:MEnergyBound}
 E_{\times}(A) \ll \log|A|\cdot \max\{|A + A|^{7/4}|A|, |A + A|^{6/5}|A|^{8/5}, |A|^{3-\delta/2}, |A + A|^{7/4}|A|^{3/2}q^{-1/4} \}.
\end{equation}
In the above estimate, one can replace the sum set $A + A$ by the difference set $A -A$.
\end{lemma}

For $X \subset \mathbb{F}_q$, we define the quotient set of $X$ by
 \begin{displaymath}
R(X) =\bigg\{\frac{x_1 - x_2}{x_3 - x_4} : x_1, x_2, x_3, x_4 \in X, x_3 \neq x_4\bigg\}.
 \end{displaymath}
We make frequent use of the following basic variant of~\cite[Lemma 2.50]{TaoVu}.
\begin{lemma}
 \label{lem:RBcard}
 Let $X \subset \mathbb{F}_q$ and $r \in \mathbb{F}_q$. If $r\not\in R(X)$, then for any nonempty subsets $X_1, X_2 \subseteq X$, we have $|X_1||X_2| = |X_1 + rX_2|.$
\end{lemma}
\begin{proof}
If two distinct pairs $(x_1, x_2), (y_1, y_2) \in X_1\times X_2$ satisfy $x_1 + rx_2 = y_1 + ry_2$, it follows that $r = (x_1 - y_1)/(y_2 - x_2) \in R(X)$. Hence if $r\not\in R(X)$, then the map $(x_1, x_2) \mapsto x_1 + rx_2$ is injective on $X_1 \times X_2$, which implies the required result.
\end{proof}

We state a corollary of Lemma~\ref{lem:RBcard}, which also appears in \cite[Corollary~2.51]{TaoVu}.
\begin{lemma}
\label{lem:RBFq}
Let $X \subset \mathbb{F}_q$ with $|X| > q^{1/2}$, then $R(X) = \mathbb{F}_q$.
\end{lemma}

The next lemma has been extracted from the proof of the main result in~\cite{RoLi}. It serves as a more favourable substitute for a similar result by Katz and Shen~\cite{KatzShen2}. A precise statement of the latter can also be found in \cite[Lemma~6]{Jones}.
\begin{lemma}
\label{lem:QuotientSetSubfield}
Let $X \subset \mathbb{F}_q$. Suppose that
 \begin{displaymath}
1+R(X) \subseteq R(X) \quad \text{and} \quad X\cdot R(X) \subseteq R(X).
 \end{displaymath}
Then $R(X)=\F_X$, where $\F_X$ denotes the subfield of $\mathbb{F}_q$ generated by $X$.
\end{lemma}

The following lemma combines \cite[Lemma~3]{BouGlib} and \cite[Lemma~2.4]{RoLi}.
\begin{lemma}
\label{lem:pivotting}
Let $X \subset \mathbb{F}_q$ and let $X^{'}$ be any subset of $X$ with $|X^{'}| \approx |X|$. If $|R(X)| \gg |X|^2$, then there exists $r\in R(X)$, such that $|X^{'} + rX^{'}| \gg |X|^2$. If $|X| >q^{1/2}$, then there exists $r \in \F_q^{*}$ such that $|X^{'} + rX^{'}| \gg q$.
\end{lemma}

Next, we recall a covering lemma, which can be found in~\cite{Shen}.
\begin{lemma}
\label{lem:ShenCovering}
Let $X, Y \subseteq \F_q$. Then, for any $0 < \epsilon < 1$, there exists a constant $C(\epsilon)$, such that at least $(1-\epsilon)|X|$ elements of $X$ can be covered by 
\[
C(\epsilon)\cdot \min\bigg\{\frac{|X+Y|}{|Y|}, \frac{|X-Y|}{|Y|}\bigg\}
\]
translates of $Y$.
\end{lemma}

The following two lemmas provide well-known variants of the Pl\"{u}nnecke-Ruzsa inequality. Both lemmas also appear in \cite{KatzShe}.
\begin{lemma}
\label{lem:PlRu}
Let $X, Y_1, \dots, Y_k \subset \F_q$. Then 
 \begin{displaymath}
|Y_1 + \cdots + Y_k| \leq \frac{|X + Y_1| \cdots |X + Y_k|}{|X|^{k-1}}.
 \end{displaymath}
\end{lemma}

\begin{lemma}
\label{lem:PlRuRefined}
Let $X, Y_1, \dots, Y_k \subset \F_q$. For any $0 < \epsilon < 1$, there exists a subset $X^{'} \subseteq X$, with $|X^{'}| \geq (1-\epsilon)|X|$ such that 
 \begin{displaymath}
|X^{'} +Y_1 + \cdots + Y_k| \ll_{\epsilon} \frac{|X + Y_1| \cdots |X + Y_k|}{|X|^{k-1}}. 
\end{displaymath}
\end{lemma}

For any nonempty sets $X, Y$ in an abelian group and any set $G\subseteq X \times Y$, we define the partial difference set of $X$ and $Y$ as 
\[
X \overset{G}{-} Y = \{x - y: (x, y) \in G \}.
\]
This notation is extended to other operations in a similar way. We recall two different formulations of the Balog-Szemer\'{e}di-Gowers theorem. Lemma~\ref{lem:BG-BSG} below is due to Bourgain and Garaev~\cite{BouGar}.
\begin{lemma}
\label{lem:BG-BSG}
Let $X,Y$ be subsets of an abelian group and $G\subseteq X \times Y$. Then, there exists $X^{'}\subseteq X$ with $|X^{'}| \gg |G|/|Y|$ such that
\begin{equation*}
|X^{'} - X^{'}| \ll \frac{|X \overset{G}{-} Y|^{4} |X|^4 |Y|^3}{|G|^5}.
\end{equation*}
\end{lemma}
See \cite[Theorem~2.29]{TaoVu} for a proof of the following formulation.
\begin{lemma}
\label{lem:BSG3P}
Let $X,Y$ be subsets of an abelian group and $G\subseteq X \times Y$. Then, there exist subsets $X^{'} \subseteq X$ and $Y^{'} \subseteq Y$ with
\begin{equation*}
|X^{'}| \gg \frac{|G|}{|Y|}\quad \text{and} \quad |Y^{'}| \gg \frac{|G|}{|X|}
\end{equation*}
such that
\[
|X^{'} + Y^{'}| \ll \frac{|X \overset{G}{+} Y|^{3} |X|^4 |Y|^4}{|G|^5}.
\]
\end{lemma}

The following lemma is due to Bourgain~\cite{Bou1}. A proof is also provided in \cite[Lemma~8]{Jones}.
\begin{lemma}
\label{lem:BourgainIntersection}
Let $X, Y \subset \F_q$ and let $M = \max_{y\in Y} |X + yX|$. Then there exist elements $x_1, x_2, x_3 \in X$ such that 
\[
|(X - x_1) \cap (x_2 - x_3) Y| \gg \frac{|Y||X|}{M}.
\]
\end{lemma}

We use the following popularity pigeonholing argument on numerous occasions throughout the paper. See \cite[Lemma~9]{Jones} for a proof.
\begin{lemma}
\label{lem:popularity}
Let $X$ be a finite set and let $f$ be a function such that $f(x) > 0$ for all $x \in X$. Suppose that 
 \begin{displaymath}
\sum_{x\in X} f(x) \geq K.
 \end{displaymath}
Let $Y = \{ x \in X : f(x) \geq K/2|X|\}.$ Then 
 \begin{displaymath}
\sum_{y\in Y} f(y) \geq \frac{K}{2}.
 \end{displaymath}
Furthermore, if $f(x) \leq M$ for all $x\in X$, then $|Y| \geq K/(2M)$.
\end{lemma}

Throughout Lemma~\ref{lem:TtoLandIPL} below, with a slight abuse of notation, we use $T(P)$ to denote the number of collinear triples formed by a set of points $P \subset \F^2_q$. Also recall that $L(P)$ denotes the number of distinct lines determined by pairs of points in $P$.
\begin{lemma}
\label{lem:TtoLandIPL}
Let $P$ be a set of points and $L$ a set of lines in $\F_q^2$. For $k\geq 1$, define
\[
I_k(P, L) = |\{(p_1, \dots , p_k, l) \in P^k \times L : p_1, \dots, p_k \in l \}|.
\]
Then, we have the inequalities
\begin{equation}
\label{eqn:KIncidenceToIncidence}
I(P, L) \leq I_k(P, L) ^{1/k}  |L|^{(k-1)/k},
\end{equation}
\begin{equation}
\label{eqn:TriplesToIncidence}
I(P, L) \ll  T(P)^{1/3}|L|^{2/3} + |L|,
\end{equation}
\begin{equation}
\label{eqn:TriplestoP}
|P|^2 \ll |L(P)|^{1/3} T(P)^{2/3}.
\end{equation}
\end{lemma}
\begin{proof}
For a line $l \in L$, we use $1_l$ to denote the indicator function of $l$. Namely, given a point $p\in \F_q^2$, we have $1_l(p) =1$ if $p\in l$ and $0$ otherwise. Then, clearly
\[
I(P, L) = \sum_{l\in L}\sum_{p \in P} 1_{l}(p).
\]
For $k\geq 1$, we obtain inequality \eqref{eqn:KIncidenceToIncidence} by an application of H\"{o}lder's inequality and the observation that
\begin{equation}
\label{IkSum}
I_k(P, L) = \sum_{l\in L}\Big(\sum_{p \in P} 1_{l}(p)\Big)^{k}.
\end{equation}
Next, we claim that $I_3(P, L) \ll I_3(P, L(P)) + |L|$. To see this, note that the contribution to $I_3(P, L)$ coming from lines containing exactly one point is bounded by $|L|$ and the contribution from lines containing two or more points of $P$ is of order $I_3(P, L(P))$. Then \eqref{eqn:TriplesToIncidence} follows from \eqref{eqn:KIncidenceToIncidence}, with $k=3$, and the simple observation that $I_3(P, L(P)) = T(P)$.

To prove \eqref{eqn:TriplestoP}, note that by H\"{o}lder's inequality, we have
\[
\sum_{l\in L(P)}\Big(\sum_{p \in P} 1_{l}(p)\Big)^{2} \leq \bigg(\sum_{l\in L(P)}\Big(\sum_{p \in P} 1_{l}(p)\Big)^{3}\bigg)^{2/3} \Big(\sum_{l\in L(P)} 1\Big)^{1/3} .
\]
Recalling identity \eqref{IkSum}, this reduces to
\[
I_2(P, L(P)) \leq I_3(P, L(P))^{2/3} L(P)^{1/3}.
\]
Then, since $I_3(P, L(P)) = T(P)$ and $I_2(P, L(P)) \gg |P|^{2}$, the required inequality follows.
\end{proof}

For sets $A, B \subset \F_q$, we define the additive energy $E_{+}(A, B)$, $E_{+}(A)$ as the additive analogue of \eqref{eqn:MEnergyDefn}. We have the following consequence of the Cauchy-Schwarz inequality
\begin{equation}
\label{eqn:AEnergyCS}
E_{+}(A, B)|A\pm B| \geq |A|^{2}|B|^{2}.
\end{equation}

The following lemma is a slight variation of a result due to Bourgain~\cite[Theorem~4.1]{Bou}. It can also be found in \cite[Lemma~14]{Shkr}.
\begin{lemma}
\label{lem:IncidencesToEnergies}
Let $A, B \subset \F_q$. Then
\begin{equation}
\label{eqn:EnergyR1}
E_{+}(1/A, 1/B) \leq I(P, L),
\end{equation}
where $P = (A + B)^{-1} \times (A + B)^{-1}$ and $L$ is a set of lines with $|L| \ll |B|^{2}$.
\end{lemma}
\begin{proof}
Let $X = (A + B)^{-1}$ and $S = (1/A + 1/B)^{-1}$. Note that elements of $S$ are of the form $ab/(a+b)$ with $a \in A$ and $b \in B$. Observing the identity
\[
\frac{1}{b} - \frac{1}{b^{2}}\cdot \frac{ab}{a+b} = \frac{1}{a+b},
\]
it follows that the cardinality
\begin{equation}
\label{eqn:SReq1}
|\{(c, d, s) \in B^{-1} \times B^{-1} \times S : (c -sc^{2}, d - sd^{2}) \in X \times X\}|
\end{equation}
can be interpreted as the number of incidences between $X\times X$ and the set of $O(|B|^{2})$ lines of the form
\begin{equation}
\label{eqn:LinesEquation}
y = \frac{d^{2}}{c^{2}}x + d\Big(1 - \frac{d}{c}\Big).
\end{equation}
Furthermore, note that if a quadruple $(a_1, a_2, b_1, b_2) \in A^{2} \times B^{2}$ satisfies
\[
\frac{1}{a_1} + \frac{1}{b_1} = \frac{1}{a_2} + \frac{1}{b_2},
\]
then $(b^{-1}_1, b^{-1}_2, a_1b_1/(a_1 + b_1))$ is a unique solution to \eqref{eqn:SReq1}.
\end{proof}

\section{Proofs of the main results}
\subsection{Incidence bounds}

\begin{proof}[Proof of Theorem \ref{thm:UnbalancedIncidences}]
First, we collect some useful tools that will be required in the proof of Theorem~\ref{thm:UnbalancedIncidences}. Claim~\ref{lem:JonesClaim1}, stated below, closely follows \cite[Claims~1 and 2]{Jones}. Although, by carrying out some calculations more efficiently, we obtain an improvement on the final conclusion of this result.
\begin{claim}
\label{lem:JonesClaim1}
Let $A, B\subseteq \F_q$ and let $T = T^{*}(A, B)$ defined in \eqref{Def:NTTriples}. There exist distinct elements $b_1, b_2 \in B$ and a set $C \subset (B - b_1)/(b_2 - B)$ with
\begin{equation}
 \label{eqn:TCSize}
|C| \gg \frac{T^{5}}{|A|^{10}|B|^{14}}
\end{equation}
such that for all $c\in C$, there exist subsets $A_{1}^{(c)}, A_{2}^{(c)} \subseteq A$ with
\begin{equation}
\label{eqn:TA1,A2Sizes}
|A_{1}^{(c)}|, |A_{2}^{(c)}| \gg  \frac{T}{|A||B|^{3}}
\end{equation}
satisfying
\begin{equation}
\label{eqn:TA1+bA2estimate}
|A_{1}^{(c)} + cA_{2}^{(c)}| \ll \frac{|A|^{11}|B|^{15}}{T^{5}}
\end{equation}
and
\begin{equation}
\label{eqn:Tbsumsetestimates}
|A_{2}^{(c)} + A_{2}^{(c)}| \ll \frac{|A|^{23}|B|^{33}}{T^{11}}.
\end{equation}
Moreover, there exists some $c_{*} \in C$ such that, writing $A_{i}^{*}$ instead of $A_{i}^{(c_{*})}$, for all $c\in C$ we have
\begin{equation}
\label{eqn:TProductofIntersections*}
|A_{1}^{(c)} \cap A_{1}^{*}||A_{2}^{(c)} \cap A_{2}^{*}| \gg \frac{T^{4}}{|A|^{6}|B|^{12}}
\end{equation}
and
\begin{equation}
\label{eqn:c*plusc}
 |c_{*}A_{2}^{*} + cA_{2}^{*}|\ll \frac{|A|^{51}|B|^{75}}{T^{25}}.
\end{equation}
\end{claim}

\begin{proof}
By the pigeonhole principle, there exist distinct elements $b_1, b_2 \in B$ such that 
\[
\bigg| \bigg\{(a_1,a_2, b) \in \ A \times A  \times (B\setminus\{b_1, b_2\}): a_1 \Big(1- \frac{b-b_1}{b_2 - b_1}\Big) + a_2\Big(\frac{b-b_1}{b_2-b_1}\Big) \in A\bigg\}\bigg| \gg \frac{T}{|B|^2}.
\]
By Lemma~\ref{lem:popularity}, there exists a set $B^{'}\subseteq B\setminus\{b_1, b_2\}$, with $|B^{'}| \gg T/(|A|^{2}|B|^{2})$, such that for each $b\in B^{'}$ there exist $\Omega(T/|B|^{3})$ pairs $(a_1, a_2) \in A \times A$, which satisfy
\[
a_1 \Big(1- \frac{b-b_1}{b_2 - b_1}\Big) + a_2\Big(\frac{b-b_1}{b_2-b_1}\Big) \in A.
\]
For $b \in B^{'}$ denote 
\[
G = \bigg\{(a_1, a_2) \in A^{2} : a_1 \Big(1- \frac{b-b_1}{b_2 - b_1}\Big) + a_2\Big(\frac{b-b_1}{b_2-b_1}\Big) \in A\bigg\},
\]
such that $|G| \gg T/|B|^{3}$. We apply Lemma~\ref{lem:BSG3P} with
\[
X = \Big(1- \frac{b-b_1}{b_2 - b_1}\Big)A, \quad Y = \Big(\frac{b-b_1}{b_2-b_1}\Big)A.
\]
Consequently, we deduce that there exist subsets $A_1^{(b)}, A_2^{(b)} \subseteq A$, with
\[
|A_1^{(b)}|, |A_2^{(b)}| \gg \frac{T}{|A||B|^{3}},
\]
which satisfy
\[
\bigg|A_1^{(b)} + \frac{b - b_1}{b_2 - b}A_2^{(b)} \bigg| \ll\frac{|A|^{11}|B|^{15}}{T^{5}}.
\]
Let 
\[
C^{'} = \bigg\{\frac{b-b_1}{b_2 -b}: b \in B^{'}\bigg\}
\]
and note that
\[
|C^{'}| \approx |B^{'}| \gg \frac{T}{|A|^{2}|B|^{2}}.
\]
Then for each $c\in C^{'}$, by a change of the indexing, we have sets $A_1^{(c)}, A_2^{(c)} \subset A$, with
\[
|A_1^{(c)} + cA_2^{(c)}| \ll \frac{|A|^{11}|B|^{15}}{T^{5}}
\]
and
\[
|A_1^{(c)}|, |A_2^{(c)}| \gg \frac{T}{|A||B|^{3}}.
\]
This gives \eqref{eqn:TA1,A2Sizes} and \eqref{eqn:TA1+bA2estimate}. Let $P_c = A_1^{(c)} \times A_2^{(c)}$. Then, for each $c\in C^{'}$, it follows that 
\[
|P_c| \gg \frac{T^2}{|A|^{2}|B|^{6}}.
\]
By the Cauchy-Schwarz inequality, we have
\[
|C^{'}|\frac{T^2}{|A|^{2}|B|^{6}} \ll \sum_{c\in C^{'}}|P_c| \leq |A|\bigg(\sum_{c_1, c_2\in C^{'}}|P_{c_1} \cap P_{c_2}|\bigg)^{1/2}.
\]
Then, by the pigeonhole principle, there exists $c_* \in C^{'}$ such that
\[
\sum_{c\in C^{'}}|P_{c} \cap P_{c_*}| \gg |C^{'}|\frac{T^{4}}{|A|^{6}|B|^{12}} \gg \frac{T^{5}}{|A|^{8}|B|^{14}}.
\]
By Lemma~\ref{lem:popularity}, there exists a subset $C \subseteq C^{'}$, with
\[
|C| \gg \frac{T^{5}}{|A|^{10}|B|^{14}},
\]
such that for each $c\in C$
\[
|P_{c} \cap P_{c_*}| \gg \frac{T^{4}}{|A|^{6}|B|^{12}}.
\]
This implies \eqref{eqn:TProductofIntersections*}, since
\[
|P_{c} \cap P_{c_*}| = |A_{1}^{(c)} \cap A_{1}^{*}||A_{2}^{(c)} \cap A_{2}^{*}|.
\]
By Lemma~\ref{lem:PlRu}, \eqref{eqn:TA1,A2Sizes} and \eqref{eqn:TA1+bA2estimate} we get
\[
|A_{2}^{(c)} + A_{2}^{(c)}| \leq \frac{|A_1^{(c)} + cA_2^{(c)}|^{2}}{|A_2^{(c)}|} \ll \frac{|A|^{23}|B|^{33}}{T^{11}},
\]
which proves \eqref{eqn:Tbsumsetestimates}. Next, by Lemma~\ref{lem:PlRu}, we have
\begin{align}
\label{eqn:TJpdc}
\nonumber |c_{*}A_{2}^{*} + cA_{2}^{(c)}| &\leq \frac{|c_{*}A_{2}^{*} + (A_{1}^{(c)} \cap A_{1}^{*})||cA_{2}^{(c)} + (A_{1}^{(c)} \cap A_{1}^{*})|}{|A_{1}^{(c)} \cap A_{1}^{*}|}
\\ &\leq \frac{|A_{1}^{*} +c_{*}A_{2}^{*} ||A_{1}^{(c)} + cA_{2}^{(c)}|}{|A_{1}^{(c)} \cap A_{1}^{*}|}.
\end{align}
Then, by Lemma~\ref{lem:PlRu} and \eqref{eqn:TJpdc} we get
\begin{align*}
|c_{*} A_{2}^{*} + cA_{2}^{*}| &\leq \frac{|c_{*} A_{2}^{*} + c(A_{2}^{(c)} \cap A_{2}^{*})||c A_{2}^{*} + c(A_{2}^{(c)} \cap A_{2}^{*})|}{|A_{2}^{(c)} \cap A_{2}^{*}|}
\\ &\leq \frac{|c_{*} A_{2}^{*} + c A_{2}^{(c)}||A_{2}^{*} +A_{2}^{*}|}{|A_{2}^{(c)} \cap A_{2}^{*}|}
\\&\leq \frac{|A_{1}^{*} + c_{*}A_{2}^{*}||A_{1}^{(c)}+ cA_{2}^{(c)}||A_{2}^{*} +A_{2}^{*}|}{|A_{1}^{(c)} \cap A_{1}^{*}||A_{2}^{(c)} \cap A_{2}^{*}|}.
\end{align*}
We obtain \eqref{eqn:c*plusc} by applying \eqref{eqn:TA1+bA2estimate}, \eqref{eqn:Tbsumsetestimates} and \eqref{eqn:TProductofIntersections*}.
\end{proof}

We use Lemma~\ref{lem:ShenCovering} and Claim~\ref{lem:JonesClaim1} to record a useful covering argument.
\begin{claim}
\label{app:TCovering}
Fix $n\leq 4$ and for $1 \leq i \leq n$, let $c_i \in C$ be arbitrary elements. Let
\begin{equation}
\label{eqn:GammaCDef}
\Gamma := \frac{|A|^{40}|B|^{60}}{T^{20}}.
\end{equation}
 Given any set $Y \subseteq A_{2}^{*}$, there exists a subset $Y^{'} \subseteq Y$ with $|Y^{'}| \approx |Y|$ such that, for $1 \leq i \leq n$, the sets $\pm c_{i}Y^{'}$ can each be fully covered by $O(\Gamma)$ translates of $A_{1}^{*}$.
\end{claim}

\begin{proof}
Fix $0 < \epsilon < 1/16$. For $1 \leq i \leq n$, by Lemma~\ref{lem:ShenCovering}, there exist sets $Y_{c_i} \subseteq Y$ with $|Y_{c_i}| \geq (1-\epsilon)|Y|$ such that $\pm c_{i}Y_{c_i}$ can be covered by
\[
O_{\epsilon}\Bigg(\frac{|\pm c_{i}Y \pm (A_{1}^{(c_{i})} \cap A_{1}^{*})|}{|A_{1}^{(c_{i})} \cap A_{1}^{*}|}\Bigg) = O_{\epsilon}\Bigg( \frac{|c_{i}A_{2}^{*} + (A_{1}^{(c_{i})} \cap A_{1}^{*})|}{|A_{1}^{(c_{i})} \cap A_{1}^{*}|}\Bigg)
\]
translates of $A_{1}^{(c_{i})} \cap A_{1}^{*} \subseteq A_{1}^{*}$. By Lemma~\ref{lem:PlRu} and the estimates \eqref{eqn:TA1+bA2estimate}, \eqref{eqn:Tbsumsetestimates} and \eqref{eqn:TProductofIntersections*} we have
\begin{align*}
\frac{|c_{i}A_{2}^{*} + (A_{1}^{(c_{i})} \cap A_{1}^{*})|}{|A_{1}^{(c_{i})} \cap A_{1}^{*}|} &\leq \frac{|c_{i}A_{2}^{*} + c_{i}(A_{2}^{(c_{i})} \cap A_{2}^{*})||(A_{1}^{(c_{i})} \cap A_{1}^{*}) +c_{i}(A_{2}^{(c_{i})} \cap A_{2}^{*})|}{|A_{1}^{(c_{i})} \cap A_{1}^{*}||A_{2}^{(c_{i})} \cap A_{2}^{*}|} \\
&\leq \frac{|A_{2}^{*} +A_{2}^{*}||A_{1}^{(c_{i})} + c_{i}A_{2}^{(c_{i})}|}{|A_{1}^{(c_{i})} \cap A_{1}^{*}||A_{2}^{(c_{i})} \cap A_{2}^{*}|}
\\ &\ll \frac{|A|^{40}|B|^{60}}{T^{20}}.
\end{align*}
Let $Y^{'} = Y_{c_1}\cap \cdots \cap Y_{c_n},$ so that $|Y^{'}| \geq (1-n\epsilon) |Y|\geq (3/4)|Y|$. Hence, for $1 \leq i \leq n$, the sets $\pm c_{i}Y^{'}$ are each fully contained in $O(\Gamma)$ translates of $A_{1}^{*}$.
\end{proof}

We apply Lemma~\ref{lem:BourgainIntersection} with $X= A_{2}^{*}$ and $Y= C/c_{*}$. Recalling \eqref{eqn:c*plusc}, we take
\[
M =O \bigg(\frac{|A|^{51}|B|^{75}}{T^{25}}\bigg).
\]
Hence, there exist elements $a_1, a_2, a_3 \in A_{2}^{*}$ such that
\[
\bigg| (A_{2}^{*} - a_1) \cap \frac{a_2 - a_3}{c_{*}}C\bigg| \gg \frac{|A_{2}^{*}||C|}{M} \gg \frac{T^{31}}{|A|^{62}|B|^{92}}.
\]
Since the conditions and the desired estimates of Theorem~\ref{thm:UnbalancedIncidences} are unchanged under dilation, without loss of generality, we assume $(a_2 - a_3) = 1$. Then, setting 
\[
D= (A_{2}^{*} - a_1) \cap (C/c_{*}),
\]
we have
\begin{equation}
 \label{eqn:DSizeLowerBound}
|D| \gg \frac{T^{31}}{|A|^{62}|B|^{92}}.
\end{equation}

We consider three cases.

\subsection*{ Case 1:} $1 + R(D) \nsubseteq R(D)$.
There exist elements $a$, $b$, $c$, $d \in C$ such that
\[
r = 1+ \frac{a - b}{c - d} \not\in R(D).
\]
By Lemma~\ref{lem:RBcard}, for any set $D^{'} \subseteq D$, with $|D^{'}| \approx |D|$, we have
\[
|D^{'}|^{2} = |D^{'} + rD^{'}|.
\]
Let $E \subseteq A_{2}^{*}$ with $|E|\approx |A_{2}^{*}|$. Applying Lemma~\ref{lem:PlRu} with $X = (c - d)E$, we get
\begin{align*}
|E||D^{'}|^{2} &= |E||(c-d)D^{'} + (c-d)D^{'} + (a-b)D^{'}| \\ &\leq |E + D^{'} + D^{'}| |cE - dE + aD^{'} - bD^{'}| \\ &\leq |A_{2}^{*} + A_{2}^{*} + A_{2}^{*}||cE - dE + aD^{'} - bD^{'}|.
\end{align*}
Now, by Claim~\ref{app:TCovering}, there exist subsets $E^{'}\subseteq A_{2}^{*}$ with $|E^{'}|\approx |A_{2}^{*}|$ and $D^{''} \subseteq D$ with $|D^{''}| \approx |D|$ such that  $cE^{'}, - dE^{'}, aD^{''}, - bD^{''}$ are contained in $O(\Gamma)$ translates of $A_{1}^{*}$. Thus, setting $E = E^{'}$ and $D^{'} = D^{''}$, we get
\[
|A_{2}^{*}||D|^{2} \ll \Gamma^{4} |A_{2}^{*} + A_{2}^{*} + A_{2}^{*}| |A_{1}^{*} + A_{1}^{*} + A_{1}^{*} + A_{1}^{*}|.
\]
By Lemma~\ref{lem:PlRu}, it follows that
\[
|D|^{2} \ll \Gamma^{4} \frac{|A_{1}^{*} + c_{*}A_{2}^{*}|^{7}}{|A_{1}^{*}|^{2}|A_{2}^{*}|^{4}}.
\]
We use \eqref{eqn:TA1,A2Sizes}, \eqref{eqn:TA1+bA2estimate}, \eqref{eqn:GammaCDef} and \eqref{eqn:DSizeLowerBound} to conclude
\[
T^{183} \ll |A|^{367}|B|^{547}.
\]

\subsection*{Case 2:} $D\cdot R(D) \nsubseteq R(D)$.
There exist elements $a$, $b$, $c$, $d$, $e \in C$ such that
\[
r = \frac{a}{c_*}\frac{b - c}{d - e} \not\in R(D).
\]
Then, for any subset $D^{'} \subseteq D$, with $|D^{'}| \approx |D|$, by Lemma~\ref{lem:RBcard}, we have
\[
|D^{'}|^{2} = |D^{'} + rD^{'}|.
\]
Let $E \subseteq A_{2}^{*}$ be any set with $|E|\approx |A_{2}^{*}|$. Applying Lemma~\ref{lem:PlRu} with $X = \frac{b-c}{d-e}E$, we get
\begin{align*}
|E||D^{'}|^{2} &= |E||D^{'} + rD^{'}| \\ &\leq \Big|D^{'} + \frac{b -c}{d -e} E \Big| |c_{*}E + aD^{'} | \\ & = |dD^{'} - eD^{'} +bE - cE| |c_{*}E + aD^{'}|.
\end{align*}
By Claim~\ref{app:TCovering}, there exist subsets $D^{''} \subseteq D$ with $|D^{''}| \approx |D|$ and $E^{'} \subseteq E$ with $|E^{'}|\approx |E|$ such that $aD^{''}, dD^{''}, -eD^{''}, bE^{'}, -cE^{'}$ are contained in $O(\Gamma)$ translates of $A_{1}^{*}$. Thus, setting $D^{'} = D^{''}$ and $E = E^{'}$, we get
\begin{align*}
|A_{2}^{*}||D|^{2} &\ll \Gamma^{5} |A_{1}^{*}+A_{1}^{*}+A_{1}^{*}+A_{1}^{*}||A_{1}^{*} + c_{*}A_{2}^{*}|\\
&\ll \Gamma^{5} \frac{|A_{1}^{*} + c_{*}A_{2}^{*}|^{5}}{|A_{2}^{*}|^{3}}.
\end{align*}
To obtain the last inequality we used Lemma~\ref{lem:PlRu}. Then, by \eqref{eqn:TA1,A2Sizes}, \eqref{eqn:TA1+bA2estimate}, \eqref{eqn:GammaCDef} and \eqref{eqn:DSizeLowerBound}, we get
\[
T^{191} \ll |A|^{383}|B|^{571}.
\]

\subsection*{Case 3:} Cases 1 and 2 do not happen. Thus, by Lemma~\ref{lem:QuotientSetSubfield} applied to the set $D$, we have $R(D) = \F_D$. Based on our assumption on the set $A$, we consider three cases.

\subsection*{Case 3.1:} $R(D) = \F_q$ and $|D| > q^{1/2}.$
Let $Y$ be an arbitrary subset of $D$ with $|Y| \approx |D|$. Then, by Lemma~\ref{lem:pivotting}, there exists an element $\xi \in \F_q^* \subset R(D)$ such that $q\ll |Y + \xi Y|$. Since $\xi \in R(D)$, there exist elements $a, b, c, d \in C$ such that
\begin{equation}
 \label{eqn:Tcase4.1pivot}
q \ll |aY - bY + cY - dY|.
\end{equation}
By Claim~\ref{app:TCovering}, there exists a subset $D^{'} \subseteq D$, with $|D^{'}| \approx |D|$ such that $aD^{'}$, $-bD^{'}$, $cD^{'}$ and $-dD^{'}$ can be covered by $O(\Gamma)$ translates of $A_{1}^{*}$. We set $Y=D^{'}$, so that by \eqref{eqn:Tcase4.1pivot} and Lemma~\ref{lem:PlRu}, we get
\[
q \ll \Gamma^{4} |A_{1}^{*}+A_{1}^{*}+A_{1}^{*}+A_{1}^{*}| \ll \Gamma^{4}\frac{|A_{1}^{*} + c_{*}A_{2}^{*}|^{4}}{|A_{2}^{*}|^{3}}.
\]
Then, by \eqref{eqn:TA1,A2Sizes}, \eqref{eqn:TA1+bA2estimate}, \eqref{eqn:GammaCDef} and \eqref{eqn:DSizeLowerBound}, we conclude that
\[
T^{103} \ll q^{-1}|A|^{207}|B|^{309}.
\]
Moreover, by Lemma~\ref{lem:RBFq}, the assumption $|D| > q^{1/2}$ implies that $R(D) = \F_q$. Hence, if $|D| > q^{1/2}$ then all other cases become impossible.
\subsection*{Case 3.2:}  Either $R(D) = \F_q$ and $|D| \leq q^{1/2}$ or $R(D)$ is a proper subfield and $|D| = |D\cap R(D)| \ll |R(D)|^{1/2}.$
Let $Y$ be an arbitrary subset of $D$ with $|Y| \approx |D|$. By Lemma~\ref{lem:pivotting}, there exist elements $a, b, c, d \in C$ such that
\begin{equation}
 \label{eqn:Tcase4.2pivot}
|D|^{2} \ll |aY - bY + cY - dY|.
\end{equation}
By Claim~\ref{app:TCovering}, there exists a subset $D^{'} \subseteq D$, with $|D^{'}| \approx |D|$ such that $aD^{'}$, $-bD^{'}$, $cD^{'}$ and $-dD^{'}$ can be covered by $O(\Gamma)$ translates of $A_{1}^{*}$. We set $Y=D^{'}$, so that by \eqref{eqn:Tcase4.2pivot} and Lemma~\ref{lem:PlRu} we get
\[
|D|^{2}  \ll \Gamma^{4} |A_{1}^{*}+A_{1}^{*}+A_{1}^{*}+A_{1}^{*}| \ll \Gamma^{4}\frac{|A_{1}^{*} + c_{*}A_{2}^{*}|^{4}}{|A_{2}^{*}|^{3}}.
\]
Then, by \eqref{eqn:TA1,A2Sizes}, \eqref{eqn:TA1+bA2estimate}, \eqref{eqn:GammaCDef} and \eqref{eqn:DSizeLowerBound}, we conclude the inequality
\[
T^{165} \ll |A|^{331}|B|^{493}.
\]

\subsection*{Case 3.3:}  $R(D)$ is a proper subfield and $|D| = |D\cap R(D)| \ll |A|^{31/191}|B|^{129/191}.$
Then, by \eqref{eqn:DSizeLowerBound}, we obtain
\[
T^{191} \ll |A|^{383}|B|^{571}.
\]

Finally, collecting the acquired bounds on $T^*(A, B)$ from the above cases, we use \eqref{eqn:T*toT} to conclude estimate \eqref{eqn:TriplesUnbalanced}. Then, we get \eqref{eqn:UnbalancedcartesianIncidence1} and \eqref{eqn:NoOfLinesUnbalanced} by subbing the acquired bound on $T(A, B)$ into the inequalities \eqref{eqn:TriplesToIncidence} and \eqref{eqn:TriplestoP} respectively.
\end{proof}

\begin{proof}[Proof of Theorem \ref{thm:balancedIncidences}]
Let $A, B \subset \F_q$ and let $T = T^{*}(A, B)$. Throughout the proof, we treat $A$ and $B$ as potentially different sets. However, our method is not particularly effective in dealing with sets of different sizes and so ultimately we assume $|A| = |B|$ to prove estimate \eqref{eqn:Triplesbalanced} under restriction \eqref{eqn:TriplesRestriction}.

By the pigeonhole principle, there exist distinct elements $b_1, b_2 \in B$ such that 
\[
\bigg| \bigg\{(a_1,a_2, b) \in \ A \times A  \times (B\setminus\{b_1, b_2\}): a_1 \Big(1- \frac{b-b_1}{b_2 - b_1}\Big) + a_2\Big(\frac{b-b_1}{b_2-b_1}\Big) \in A\bigg\}\bigg| \gg \frac{T}{|B|^2}.
\]
Let 
\[
B_{*} = \bigg\{\frac{b-b_1}{b_2 -b_1}: b \in B\setminus\{b_1, b_2\}\bigg\}.
\]
By Lemma~\ref{lem:popularity} there exists a set $A_{*}\subseteq A$, with $|A_{*}| \gg T/(|A||B|^3)$, such that
\begin{equation}
\label{mainpopularity}
|\{(a_1,a_2, b) \in A \times A_{*} \times B_{*} : a_1(1- b) + a_2 b \in A\}| \gg \frac{T}{|B|^2}
\end{equation}
and for each $a\in A_{*}$ we have 
\begin{equation}
\label{apopularity}
|\{(a_1, b) \in A \times B_{*} : a_1(1-b) + ab \in A\}| \gg \frac{T}{|A||B|^2}.
\end{equation}
By the pigeonhole principle, applied to \eqref{mainpopularity}, there exists an element $b_0 \in B_{*}$ such that
\begin{equation}
\label{secondG}
\big| \{(a_1,a_2) \in A \times A_{*}: a_1( 1- b_{0}) + a_2 b_{0} \in A\} \big| \gg \frac{T}{|B|^3}.
\end{equation}
We apply Lemma~\ref{lem:BG-BSG} with
\[
X=b_{0} A_{*}, \quad Y = (b_{0} - 1)A\quad \text{and} \quad G = \{(x, y) \in X \times Y : x - y \in A\}.
\]
Observing that
\[
|X| \leq |A|, \quad |Y|  = |A|, \quad |G| \gg \frac{T}{|B|^3} \quad \text{and} \quad  |X\overset{G}{-} Y|  \leq |A|,
\]
we deduce that there exists a subset $A^{'} \subseteq A_{*}$, with
\begin{equation}
\label{Aprimesize}
|A'|\gg \frac{|G|}{|Y|} \gg \frac{T}{|A||B|^3},
\end{equation} 
such that
\begin{equation}
\label{diffbound}
|A' - A'| \ll \frac{|X \overset{G}{-} Y|^{4} |X|^4 |Y|^3}{|G|^5} \ll \frac{|A|^{11}|B|^{15}}{T^{5}}.
\end{equation} 
Since $A'\subseteq A_{*}$, by \eqref{apopularity} and \eqref{Aprimesize}, it follows that
\begin{align}
\label{Aprimesub}
|\{(a_1,a_2, b) \in A \times A^{'} \times B_{*} : a_1(1-b) + a_2 b \in A\}| \gg \frac{T^2}{|A|^2|B|^5}.
\end{align}
By the pigeonhole principle, applied to \eqref{Aprimesub}, there exists an element $a_0\in A$ such that 
\[
| \{(a, b) \in A^{'}\times B_{*} : a_0 (1-b) + ab \in A\}| \gg \frac{T^2}{|A|^3|B|^5}.
\]
Equivalently,
\begin{equation}
\label{b2eqyq}
| \{(a, b) \in (A^{'}-a_0) \times B_{*} : ab \in (A - a_0)\}| \gg \frac{T^2}{|A|^3|B|^5}.
\end{equation}
We use Lemma~\ref{lem:BG-BSG}, multiplicatively, with 
\[
X = A^{'} - a_0, \quad Y = B_{*}^{-1}\quad \text{and} \quad G = \{(x, y) \in X \times Y: x/y \in (A - a_0) \}.
\]
Observe that
\[
 |X| \leq|A|, \quad |Y| \approx |B|, \quad |G|\gg \frac{T^2}{|A|^3|B|^5} \quad \text{and} \quad |X\overset{G}{/}Y| \leq |A|.
\]
We conclude that there exists a subset $C \subseteq A^{'} - a_0$, with
\begin{equation}
\label{Csize}
|C|\gg \frac{|G|}{|Y|} \gg \frac{T^2}{|A|^3|B|^6},
\end{equation} 
such that
\begin{equation}
\label{prodbound}
|C / C| \ll \frac{|X \overset{G}{/} Y|^{4} |X|^4 |Y|^3}{|G|^5} \ll \frac{|A|^{23}|B|^{28}}{T^{10}}.
\end{equation} 
Since $C\subseteq A' - a_0$, by \eqref{diffbound}, we also have
\begin{equation}
|C - C| \ll \frac{|A|^{11}|B|^{15}}{T^{5}}.
\end{equation} 
By Lemma~\ref{lem:sum-ratio}, applied to the set $C$, we get
\[
T^{*}(A, B) \ll |A|^{217/104} |B|^{302/104} + q^{-1/95}|A|^{199/95}|B|^{277/95}.
\]
To obtain \eqref{eqn:Triplesbalanced}, we set $B =A$ and use \eqref{eqn:T*toT}. It remains to justify our use of Lemma~\ref{lem:sum-ratio}. Fix an arbitrary $0<\eta < 1/8$. Then, for any constant $\lambda>0$, we may assume $\lambda |A|^{51/52} \leq \eta|C|$ as otherwise, recalling \eqref{Csize}, we can use the lower bound $|C| \gg T^{2}/|A|^{9}$ to obtain the required estimate. Now, suppose the set $A$ satisfies condition \eqref{eqn:TriplesRestriction}. Then, given an arbitrary proper subfield $G$ and element $c \in \F_q$, there exists some constant $\lambda >0$ such that
\[
|C\cap cG| \leq |(A-a_0) \cap cG| \leq \lambda\cdot \max\{|G|^{1/2}, |A|^{51/52}\} \leq \max\{\lambda|G|^{1/2}, \eta |C|\},
\]
as required. Finally, we obtain \eqref{eqn:cartesianIncidence1} and \eqref{eqn:NoOfLines} by subbing the acquired bound on $T(A)$ into the inequalities \eqref{eqn:TriplesToIncidence} and \eqref{eqn:TriplestoP} respectively.
\end{proof}
\begin{proof}[Proof of Theorem~\ref{Thm:TSumsetBound}]
Recall the definitions \eqref{triplesDefn} and \eqref{eqn:MEnergyDefn}. The number of collinear triples formed by $A\times A$ can be expressed as
\[
T(A) = \sum_{a,b \in A} E_{\times}(A + a, A + b).
\]
By \eqref{eqn:MEECS} and another application of the Cauchy-Schwarz inequality, we get
\[
T(A) \leq \bigg(\sum_{a \in A} E_{\times}(A + a)^{1/2}\bigg)^2 \leq |A|\sum_{a \in A} E_{\times}(A + a) \leq|A|^{2}\max_{a\in A} E_{\times}(A + a).
\]
Then, under restriction \eqref{eqn:TSmallSumsetRestriction}, we can bound $\max_{a\in A} E_{\times}(A + a)$ using Lemma~\ref{lem:energybound}. This concludes the proof of \eqref{eqn:SumsetToTriples}. 
\end{proof}
\begin{proof}[Proof of Corollary~\ref{cor:SmSSIPL}]
Under restriction \eqref{eqn:TSmallSumsetRestriction}, with $\delta=2/5$, we use \eqref{eqn:SumsetToTriples} and the assumption $|A+A|\approx |A|$ to get \eqref{eqn:CTIntervals}. Then, using this bound on $T(A)$ together with \eqref{eqn:TriplesToIncidence} and \eqref{eqn:TriplestoP}, we obtain the estimates \eqref{eqn:IncidencesIntervals} and \eqref{eqn:LinesIntervals} respectively.
\end{proof}

\subsection{Applications}
\begin{proof}[Proof of Corollary~\ref{cor:SR}]
We use Lemma~\ref{lem:IncidencesToEnergies}, together with the incidence bound \eqref{eqn:cartesianIncidence1} of Theorem~\ref{thm:balancedIncidences}, to deduce \eqref{eqn:ERABFq}. Then, we set $B = A$ and apply \eqref{eqn:AEnergyCS} to get
\[
\frac{|A|^{4}}{|1/A + 1/A|} \ll \Big(|A + A|^{173/104} + q^{-1/285}|A+A|^{476/285}\Big) |A|^{4/3}.
\]
Hence, either
\[
|A+A|^{519}|1/A + 1/A|^{312} \gg |A|^{832}
\]
or
\[
|A+A|^{476}|1/A + 1/A|^{285} \gg q|A|^{760}.
\]
This gives \eqref{eqn:SRvsSS}. Next, we set $B = A^{-1}$ and use \eqref{eqn:AEnergyCS} to get
\[
\frac{|A|^{4}}{|A + 1/A|} \ll \Big(|A + 1/A|^{173/104} + q^{-1/285}|A + 1/A|^{476/285}\Big) |A|^{4/3},
\]
which gives estimate \eqref{eqn:SRExpander}.

Now, assume that $\F_q$ does not contain any proper subfields $G$, with
\begin{equation}
\label{eqn:corSRd1}
|A|^{51/52} < |G| < |A|^{(2\cdot 832)/831}.
\end{equation}
Then, suppose that for some proper subfield $G$ and elements $c,d \in \F_q$, the left hand side of \eqref{eqn:ERCon} is larger than $|G|^{1/2}$. By our assumption \eqref{eqn:corSRd1}, either $|G| \leq |A|^{51/52}$ so that restriction \eqref{eqn:ERCon} is satisfied or $|G|\geq |A|^{(2\cdot 832)/831}$ so that we have
\[
|A+B| \geq |G|^{1/2} \geq |A|^{1+1/831}.
\]
This gives the relevant required estimates for both cases $B =A$ and $B = A^{-1}$. Finally, we restate \eqref{eqn:corSRd1} as a condition on $A$, by asking that the cardinality of $A$ does not lie in the intervals $(|G|^{1/2-1/1664}, |G|^{1+1/51})$ for all proper subfields $G$ of $\F_q$.
\end{proof}

\begin{proof}[Proof of Corollary~\ref{cor:Hyperbolic}]
Let $S = A \cap \alpha/A$. Suppose that $S$ satisfies condition \eqref{eqn:ERCon}. Then, noting that $S + S \subseteq A + A$ and $\alpha/S + \alpha/S \subseteq A + A$, by the estimate \eqref{eqn:SRvsSS} we have
\[
|S| \ll  |A + A| ^{831/832} + q^{-1/760}|A+A|^{761/760}.
\]
Now, suppose that for all proper subfields $G$ and elements $c, d \in \F_q$ we have
\begin{equation}
\label{eqn:HypSolConPf}
|(A+A)^{-1} \cap (cG+d)| \ll \max\{|G|^{1/2}, |A+A|^{(831\cdot 51)/(832\cdot 52)}\}.
\end{equation}
It follows that either $|S| < |A+A|^{831/832}$, which gives the desired result or, by \eqref{eqn:HypSolConPf}, we can deduce that $S$ satisfies condition \eqref{eqn:ERCon}. Finally, note that $47/48 < (831\cdot 51)/(832\cdot 52)$, which means that condition \eqref {eqn:HypSolConPf} is satisfied under condition \eqref{eqn:HypSolCon}.

Suppose that for all proper subfields $G$ of $\F_q$, $|A| \not\in (|G|^{1/2 - 1/1664}, |G|^{1+1/47})$. Hence $\F_q$ does not contain any proper subfields $G$ with
\[
|A|^{47/48} < |G| < |A|^{(2\cdot 832)/831}.
\]
Then if, for some $c,d\in \F_q$ and a proper subfield $G$, the left hand side of \eqref{eqn:HypSolConPf} is larger than $|G|^{1/2}$, it follows that either $|G| \leq |A|^{47/48}$ so that \eqref{eqn:HypSolCon} is satisfied or
\[
|A+A| \geq |G|^{1/2} \geq |A|^{832/831}\geq |S|^{832/831},
\]
which gives the required estimate.
\end{proof}

\begin{proof}[Proof of Corollary~\ref{cor:SRI}]
Let $X = (A +A)^{-1} = 1/A$ and let $L$ be any set of lines. Note that, using estimate \eqref{eqn:SumsetToTriples} together with \eqref{eqn:TriplesToIncidence}, we can bound $I(X\times X, L)$ in terms of $|X + X|$. We use identity \eqref{eqn:EnergyR1} of Lemma~\ref{lem:IncidencesToEnergies}, as well as \eqref{eqn:AEnergyCS}, to deduce
\begin{align*}
\frac{|A|^{4}}{|1/A + 1/A|} &\leq E_{+}(1/A) \ll (\log|A|)^{1/3}\cdot \max \big\{|1/A + 1/A|^{7/12}|A|, \\ &|1/A + 1/A|^{6/15}|A|^{18/15}, |A|^{5/3-\delta/6}, |1/A + 1/A|^{7/12}|A|^{7/6}q^{-1/12}\big\} |A|^{4/3}.
\end{align*}
Choosing $\delta=3/7$, we obtain the required result based on the last three terms of the above inequality and note that the first term yields a better bound than required. Given this choice of $\delta$, by \eqref{eqn:TSmallSumsetRestriction}, we see that \eqref{eqn:SoRICondition} gives the necessary restriction on $A^{-1}$. Finally, if we assume that there are no proper subfields with
\[
|A|^{4/7} < |G| < |A|^{2},
\]
then it is necessarily the case that restriction \eqref{eqn:SoRICondition} is satisfied for all proper subfields. Since \eqref{eqn:SoRICondition} fails only if there exist some elements $c, d \in \F_q$ and a proper subfield $G$ such that
\[
|A| \geq |A^{-1} \cap (cG +d)| > |G|^{1/2} \geq |A|,
\]
which is impossible. This concludes the proof of the required lower bound on $|1/A + 1/A|$. A similar argument gives the same bound for $|1/A - 1/A|$. To see this note that when applying \eqref{eqn:SumsetToTriples} and \eqref{eqn:AEnergyCS}, one may replace $X+X$ by $X-X$.
\end{proof}
\appendix

\section{Proofs of Lemma~\ref{lem:sum-ratio} and Lemma~\ref{lem:energybound}}
\label{Appendix:SumRatio}
\begin{proof}[Proof of Lemma~\ref{lem:sum-ratio}]
Fix some $\epsilon = \epsilon(\eta)$ satisfying 
\begin{equation}
\label{eqn:EpsEta}
\eta < (1-\epsilon)/8 < 1/8.
\end{equation}
We apply Lemma~\ref{lem:PlRuRefined} to identify some subset $B \subseteq A$ with 
\begin{equation}
\label{eqn:BLB}
|B|\geq (1-\epsilon) |A|,
\end{equation} 
such that 
\begin{equation}
\label{eqn:iteratedsumset}
|B+B+B+B| \ll_{\epsilon} \frac{|A+A|^{3}}{|A|^{2}}.
\end{equation}
We point out that in order to prove the estimates of Lemma~\ref{lem:sum-ratio} involving the difference set $A -A$ instead of the sum set $A+A$, by an alternative use of Lemma~\ref{lem:PlRuRefined}, one can identify some subset $B^{'} \subseteq A$ with $|B^{'}| \geq (1-\epsilon) |A|$ such that
\begin{equation}
\label{eqn:iterateddifferenceset}
|B^{'}-B^{'}-B^{'}-B^{'}| \ll_{\epsilon} \frac{|A-A|^{3}}{|A|^{2}}.
\end{equation}
For $X\subseteq \F_q$ and $\xi \in X/X$, we write $r_{X}(\xi) = |\{ (a, b) \in X^{2} : b/a = \xi\}|.$ Note that
\[
\sum_{\xi \in B/B} r_B(\xi) = |B|^2.
\]
Let $Y\subseteq B/B$ be the set of popular slopes such that for $\xi \in Y$ we have 
\[
r_{B}(\xi) \geq \frac{|B|^2}{2|B/B|}
\]
and let $P = \{ (x, y) \in B \times B : y/x \in Y \}.$ By Lemma~\ref{lem:popularity} with $X = B/B$ and $f = r_{B}$, it follows that 
\[
|P| = \sum_{\xi \in Y}r_{B}(\xi) \geq \frac{|B|^2}{2}.
\]
By the pigeonhole principle, there exists some $x_{*}\in B$, such that the set 
\[
B_{x_*} = \{y : (x_*, y) \in P \}
\]
has cardinality $|B_{x_*}| \geq |B|/2.$ For $\xi \in \F_q$, we write $P_{\xi} = \{x :( x, \xi x)\in P\}.$
Then, for all $y \in B_{x_{*}}$, we have 
\begin{equation}
\label{eqn:ProjectionsSizeandProperty}
|P_{y/x_*}|  = r_B(y/x_*) \geq \frac{|B|^2}{2|B/B|} \quad \text{and} \quad \frac{y}{x_{*}}P_{y/x_{*}} \subseteq B.
\end{equation}
By Lemma~\ref{lem:popularity}, with $X = B_{x_{*}}/B_{x_{*}}$ and $f = r_{B_{x_{*}}}$, there exists $S \subseteq B_{x_{*}} \times B_{x_{*}}$, with 
\[
|S| \geq \frac{|B_{x_{*}}|^2}{2} \geq \frac{|B|^2}{8},
\]
such that if $(x, y) \in S$, then 
\[
 r_{B_{x_{*}}}(y/x) \geq \frac{|B_{x_{*}}|^2}{2|B_{x_{*}}/B_{x_{*}}|} \geq \frac{|B|^2}{8|B/B|}.
\]
By the pigeonhole principle there exists a popular abscissa $x_0$, such that the set $B_{x_0} = \{y: (x_0, y) \in S\}$ has cardinality 
\begin{equation}
\label{B1size}
|B_{x_0}| \geq \frac{|B|}{8}.
\end{equation}
Since the required estimate and the conditions of Lemma~\ref{lem:sum-ratio} are invariant under dilations of the set $A$, we may assume, without loss of generality, that $x_0 = 1$.
Let
\begin{equation}
\label{eqn:SyDef}
S_y = \{x : (x, xy)\in S\}
\end{equation}
and note that for any $y\in B_1$, we have 
\begin{equation}
\label{Sy}
S_y, yS_y \subseteq B_{x_{*}} \subseteq B 
\end{equation}
and that 
\begin{equation}
\label{SyLB}
|S_y| \gg \frac{|B|^2}{|B/B|}.
\end{equation}

Next, we record a useful consequence of Lemma~\ref{lem:ShenCovering}.
\begin{claim}
\label{app:Covering}
Let $n \leq 4$ and for $1\leq i \leq n$, let $b_i \in \pm B_1$ be arbitrary elements. Let
\[
\Gamma := O\bigg(\frac{|A+A||A/A|}{|B|^{2}}\bigg).
\]
Then, for any subset $C \subset B$, a further subset $C_{1} \subset C$ can be identified, with $|C_{1}| \approx |C|$, such that all of $b_{i}C_{1}$ are fully covered by $O(\Gamma)$ translates of $B$.
\end{claim}
\begin{proof}
By Lemma~\ref{lem:ShenCovering}, for any $0 < \delta \leq 1/16$ and $1\leq i \leq n$, there exist sets $C_{b_i} \subset C$ with $|C_{b_i}| \geq (1-\delta)|C|$ such that $b_{i}C_{b_i}$ can be covered by
\begin{equation}
\label{eqn:Noftranslates}
O_{\delta}\bigg(\frac{|b_{i}C + b_{i}S_{b_i}|}{|b_{i}S_{b_i}|}\bigg) = O_{\delta}\bigg(\frac{|A+A||A/A|}{|B|^{2}}\bigg)
\end{equation}
translates of $B$. We set $C_{1}= C_{b_1} \cap\dots \cap C_{b_n}$. Then, it follows that $|C_{1}| \geq (1-n\delta)|C| \geq (3/4)|C|$ and  each of $b_i C_{1}$ gets fully covered by $O(\Gamma)$ translates of $B$.
\end{proof}
We remark that by an alternative use of Lemma~\ref{lem:ShenCovering}, one may replace $A+A$ in \eqref{eqn:Noftranslates} by $A-A$. By this observation and~\eqref{eqn:iterateddifferenceset}, the remainder of the proof may be easily reworked to produce the same estimates involving the difference set.

First, we assume $|A| \ll q^{1/2}$ and consider four cases corresponding to the nature of the quotient set $R(B_1)$.
\subsection*{Case 1:}  $R(B_1) \not= R(B_{x_*})$.
Recall that $B_1 \subseteq B_{x_*}$, which implies that $R(B_1) \subseteq R(B_{x_*})$. Therefore, according to the assumption of this case, there exists some element $r\in R(B_{x_*})$ such that $r\not\in R(B_1)$. Since $r\not\in R(B_1)$, given an arbitrary subset $Y\subseteq B_1$, by Lemma~\ref{lem:RBcard}, we have $|Y|^{2} = |Y + rY|.$ Namely, there exist elements $a, b, c, d \in B_{x_*}$, such that
\begin{equation}
\label{eqn:SRC1}
|Y|^2 = | cY -  dY + aY- bY|.
\end{equation}
By Lemma~\ref{lem:ShenCovering} and \eqref{eqn:ProjectionsSizeandProperty}, a positive proportion of $c B_1$ can be covered by at most
\begin{equation*}
\frac{|c B_1 + cP_{c/x_{*}}\big|}{|P_{c/x_{*}}|} \ll \frac{|A + A||A/A|}{|B|^{2}}
\end{equation*}
translates of $c P_{c/x_{*}} \subset x_{*} B$. Similarly, a positive proportion of each of $-d B_1$, $a B_1$ and $-b B_1$ can be covered by $O(\Gamma)$ translates of $x_{*} B$.

Proceeding in a similar manner as Claim~\ref{app:Covering}, an appropriate subset $B^{'}_1 \subset B_1$ of size $|B^{'}_1 | \approx |B_1|$ can be identified, so that $cB^{'}_1$, $-dB^{'}_1$, $aB^{'}_1$ and $-bB^{'}_1$ are each fully covered by $O(\Gamma)$ translates of $x_{*}B$. Hence, by \eqref{eqn:SRC1}, with $Y = B^{'}_1$, we deduce
\begin{equation}
\label{eqn:Case1LongSum}
|B|^{2} \ll \frac{|B + B + B + B||A +A|^{4}|A/A|^{4}}{|B|^{8}}.
\end{equation}
After applying~\eqref{eqn:iteratedsumset}, it follows that
\begin{equation}
\label{eqn:Case1FinalBound}
|A+A|^{7}|A/A|^{4} \gg_{\epsilon} |A|^{12}.
\end{equation}

\subsection*{Case 2:}  $1 + R(B_1) \not\subseteq R(B_1)$.
There exist elements $a, b, c, d \in B_1$, such that
\[
r = 1 + \frac{a - b}{c - d} \not\in R(B_1) = R(B_{x_*}).
\]
Recall the set $S_a$ defined in \eqref{eqn:SyDef}. Let $S^{'}_a \subset S_a$ denote an arbitrary subset with $|S^{'}_a| \approx |S_a|$. Also let $B^{'}_1$ be an arbitrary subset of $B_1$ with $|B^{'}_1| \approx |B_1|$. By Lemma~\ref{lem:PlRuRefined}, with $X = (c-d)B^{'}_1$, there exists $B^{''}_1 \subseteq B^{'}_1$, with $|B^{''}_1| \approx |B^{'}_1| \approx |B|$, such that
\begin{align}
\label{eqn:Case2PlRu}
|B^{''}_1 + rS^{'}_a| &= |(c-d)B^{''}_1 + (c-d)S^{'}_a + (a-b)S^{'}_a| \\
\nonumber &\ll \frac{|B^{'}_1 + S^{'}_a|}{|B^{'}_1|}|(c-d)B^{'}_1 + (a-b)S^{'}_a|.
\end{align}
Recalling that $B^{''}_1, S^{'}_a \subset B_{x_*}$, by Lemma~\ref{lem:RBcard} and \eqref{SyLB}, we have 
\[
|B^{''}_1 + rS^{'}_{a}| =  |B^{''}_1||S^{'}_{a}| \gg \frac{|B|^{3}}{|A/A|}.
\]
Then, by~\eqref{eqn:Case2PlRu}, it follows that
\begin{equation}
\label{eqn:Case2PreCovering}
\frac{|B|^{4}}{|A/A|} \ll |A+A||cB^{'}_1 -dB^{'}_1 + aS^{'}_{a} - bS^{'}_{a}|.
\end{equation}
Now, by Claim~\ref{app:Covering}, we can identify positively proportioned subsets $C_1 \subset B_1$ and $C_2 \subset S_a$ such that each of $cC_1$, $-dC_1$ and $-bC_2$ gets fully covered by $O(\Gamma)$ translates of $B$. We fix $B^{'}_1 = C_1$ and $S^{'}_a = C_2$.  Then by~\eqref{eqn:Case2PreCovering} and a threefold use of Claim~\ref{app:Covering} we have
\[
|B|^{10} \ll |A/A|^{4}|A+A|^{4}|B + B + aS^{'}_a + B|.
\]
Observing that $aS^{'}_a \subset B$, then applying~\eqref{eqn:iteratedsumset}, we conclude that
\[
|A+A|^{7}|A/A|^{4} \gg_{\epsilon} |A|^{12}.
\]

\subsection*{Case 3:}  $B_{1}\cdot R(B_1) \not\subseteq R(B_1)$.
There exist elements $a, b ,c ,d ,e ,f \in B_1$ such that
\[
r = a\frac{c - d}{e - f} \not\in R(B_1 )  = R(B_{x_{*}}).
\]
Given any set $Y_1 \subseteq B_{x_*} $, recalling that $S_a \subset B_{x_*}$, by Lemma~\ref{lem:RBcard}, we have $|Y_1||S_a| = |Y_1 + rS_a|.$ For an arbitrary set $Y_2$, by Lemma~\ref{lem:PlRu} with $X = \frac{c - d}{ e -f }Y_2$, we obtain 
\begin{align*}
|Y_2||Y_1||S_a| &= |Y_2||Y_1 + rS_a| \\ &\leq \Big|Y_1 + \frac{c-d}{e-f}Y_2\Big| |a S_a + Y_2 |\\ &= |eY_1 - fY_1 + cY_2 - dY_2||aS_a + Y_2|.
\end{align*}
By Claim~\ref{app:Covering}, there exist positively proportioned subsets $C_1 \subset S_e$ and $C_2 \subset S_c$ such that $-fC_1$ and $-dC_2$ can be covered by $O(\Gamma)$ translates of $B$. We fix $Y_1 = C_1$ and $Y_2 = C_2$. First, noting that $eS_e$, $aS_a$, $cS_c \subset B$, we have
\[
|Y_1||Y_2||S_a| \ll |B - fY_1 + B - dY_2||A +A|.
\]
Then, the covering argument, together with \eqref{SyLB}, yields
\[
\frac{|B|^{6}}{|A/A|^{3}} \ll \frac{|B+B+B+B||A+A|^{3}|A/A|^{2}}{|B|^{4}}.
\]
Finally, applying~\eqref{eqn:iteratedsumset}, it follows that
\[
|A+A|^{6}|A/A|^{5} \gg_{\epsilon} |A|^{12}.
\]
\subsection*{Case 4:}  Cases 1-3 do not happen.
By Lemma~\ref{lem:QuotientSetSubfield}, it follows that $R(B_1) = \F_{B_1}$. Based on the assumptions of Lemma~\ref{lem:sum-ratio}, we consider two cases.

\subsection*{Case 4.1:} $|A \cap R(B_1)| \ll |R(B_1)|^{1/2}$.
Clearly $B_1 \subseteq \F_{B_1} = R(B_1)$. Hence 
\[
|B_1|^2 =|B_1 \cap R(B_1)|^2\leq |A \cap R(B_1)|^2 \ll |R(B_1)|.
\]
Thus, by Lemma~\ref{lem:pivotting}, there exist elements $a, b, c, d \in B_1$ such that for any subset $Y \subseteq B_1$ with $|Y| \approx |B_1|,$ we have
\[
|B|^2 \approx |Y|^2 \ll |aY -bY +cY -dY|. 
\]
By Claim~\ref{app:Covering}, there exists a subset $B_1^{'}\subset B_1$, with $|B_1^{'}| \approx |B_1|$, such that $aB_1^{'}$, $-bB_1^{'}$, $cB_1^{'}$ and $-dB_1^{'}$ are each fully contained in $O(\Gamma)$ translates of $B$. We set $Y = B^{'}_1$, to obtain
\[
|B|^{2} \ll |B +B +B +B|\frac{|A+A|^{4}|A/A|^{4}}{|B|^{8}}.
\]
Then by \eqref{eqn:iteratedsumset} we have
\[
|A+A|^{7}|A/A|^{4}\gg_{\epsilon} |A|^{12}.
\]

\subsection*{Case 4.2:} $R(B_1)$ is a proper subfield and $|A \cap R(B_1)| \leq \eta|A|$ for some fixed $0<\eta < 1/8$.
In particular, we have 
\[
|B_1| =|B_1 \cap R(B_1)| \leq |A \cap R(B_1)| \leq \eta|A|.
\]
On the other hand, by \eqref{eqn:EpsEta}, \eqref{eqn:BLB} and \eqref{B1size} we have $|B_1| > \eta |A|$. Hence this case is impossible.

Finally, suppose that $|A| \geq \eta^{-1}q^{1/2}$. Then, by \eqref{eqn:EpsEta}, \eqref{eqn:BLB} and \eqref{B1size} we have $|B_1| > q^{1/2}$. By Lemma~\ref{lem:RBFq}, it follows that $R(B_1) = \F_q$. Let $Y$ denote an arbitrary subset of $B_1$ with $|Y| \approx |B_1|.$ By Lemma \ref{lem:pivotting}, there exists an element $\xi \in\F_q^{*}$ such that $q\ll |Y + \xi Y|.$ Since $R(B_1) = \F_q$, there exist elements $a, b, c, d \in B_1$ such that 
\[
q \ll |aY -bY +cY -dY|.
\]
By Claim~\ref{app:Covering}, there exists a positively proportioned subset $B_1^{'}\subset B_1$, such that $aB_1^{'}$, $-bB_1^{'}$, $cB_1^{'}$ and $-dB_1^{'}$ are each fully covered by $O(\Gamma)$ translates of $B$. We set $Y = B^{'}_1$, to obtain
\[
q \ll |B +B +B +B|\frac{|A+A|^{4}|A/A|^{4}}{|B|^{8}}.
\]
By~\eqref{eqn:iteratedsumset}, we conclude
\[
|A+A|^{7}|A/A|^{4}\gg_{\epsilon} q|A|^{10}.
\]
\end{proof}

\begin{proof}[Sketch of proof of Lemma~\ref{lem:energybound}]

Following the proof of \cite[Theorem~1.4]{RoLi}, there exist integers $L$ and $N$, with $N < |A|$, $LN < |A|^{2}$ and such that
\begin{equation}
\label{eqn:MBound}
M:= LN^{2} \gg \frac{E_{\times}(A)}{\log|A|}.
\end{equation}
In particular, it follows
\[
L, N > \frac{M}{|A|^{2}}.
\]
By \cite[Lemma~3.1]{RoLi}, there exists a set $\tilde{A}_{x_0} \subset A$ with
\begin{equation}
\label{eqn:tildeAx0Size}
|\tilde{A}_{x_0}| \gg \frac{LM}{|A|^{3}}> \frac{M^{2}}{|A|^{5}}.
\end{equation}
Then, based on the nature of the quotient set $R(\tilde{A}_{x_0})$, five cases are considered.
Let 
\[
K=\min\bigg\{\frac{|A+A|}{|A|}, \frac{|A-A|}{|A|}\bigg\}.
\]
Case 1 and Case 2 lead to the estimate
\begin{equation}
\label{eqn:RoLiCase1Estimate}
M^{4} \ll K^{7}|A|^{11},
\end{equation}
in Case~3 we get
\begin{equation}
\label{eqn:Case3EnergyBound}
M^{3} \ll K|A|^{8}
\end{equation}
 and Case~4 yields
\begin{equation}
\label{eqn:Case4EnergyBound}
M^{5} \ll K^{6}|A|^{14}.
\end{equation}

In Case~5, it follows that $R(\tilde{A}_{x_0}) = \F_{\tilde{A}_{x_0}}$. Here we proceed slightly differently from the proof of~\cite[Theorem 1.4]{RoLi} and split Case~5 into three cases.
 
{\bf Case 5.1:} $R(\tilde{A}_{x_0}) = \F_q$ and $|\tilde{A}_{x_0}| > q^{1/2}$. Then, Lemma~\ref{lem:pivotting} can be used in conjunction with~\cite[Application~3.2]{RoLi} to obtain the estimate
\begin{equation}
\label{eqn:Case5.1EnergyBound}
M^{4} \ll q^{-1}K^{7}|A|^{13}.
\end{equation}
Furthermore, if $|\tilde{A}_{x_0}| > q^{1/2}$ by Lemma~\ref{lem:RBFq} it follows that $R(\tilde{A}_{x_0}) = \F_q$. Then, under this assumption, Cases 2-4 become impossible and Case~5.1 becomes the only possibility for Case~5.

{\bf Case 5.2:}   Either $R(\tilde{A}_{x_0}) = \F_q$ and $|\tilde{A}_{x_0}| \leq q^{1/2}$ or $R(\tilde{A}_{x_0})$ is a proper subfield and $|A \cap R(\tilde{A}_{x_0})| \ll |R(\tilde{A}_{x_0})|^{1/2}$. This case is dealt with in Case~5 of the proof of~\cite[Theorem 1.4]{RoLi}, where estimate~\eqref{eqn:RoLiCase1Estimate} is recovered.

{\bf Case 5.3:} $R(\tilde{A}_{x_0})$ is a proper subfield and $|A \cap R(\tilde{A}_{x_0})| \ll |A|^{1-\delta}$ for some fixed $\delta >0$. Then, by \eqref{eqn:tildeAx0Size}, we have
\[
\frac{M^{2}}{|A|^{5}} < |\tilde{A}_{x_0}| = |\tilde{A}_{x_0}\cap R(\tilde{A}_{x_0})| \leq |A\cap R(\tilde{A}_{x_0})| \ll |A|^{1-\delta}.
\]
This gives
\begin{equation}
\label{eqn:Case5.3EnergyBound}
M \ll |A|^{3-\delta/2}.
\end{equation}

Finally, putting together \eqref{eqn:MBound}, \eqref{eqn:RoLiCase1Estimate}, \eqref{eqn:Case4EnergyBound}, \eqref{eqn:Case5.1EnergyBound} and \eqref{eqn:Case5.3EnergyBound}, we obtain the required bound on $E_{\times}(A)$. Clearly~\eqref{eqn:Case3EnergyBound} is a stronger bound than required.
\end{proof}

\section*{Acknowledgments}
The author would like to thank Igor Shparlinski for his helpful comments on the manuscript.

\end{document}